\documentclass{amsart}
\usepackage{amsmath} 
\usepackage{amssymb}
\usepackage{amsfonts}
\usepackage{bm}
\usepackage{fixmath}
\usepackage{amsthm}
\usepackage{color}
\usepackage{MnSymbol}
\usepackage{stmaryrd}
\usepackage[utf8]{inputenc}
\usepackage{titlesec}
\usepackage[margin=2.5cm]{geometry}
\usepackage[T1]{fontenc}

\usepackage{enumitem}
\setlist{itemsep=1em,topsep=1em}

\usepackage{mathtools}
\theoremstyle{plain}

\begin{document}

\title{A Counterexample to an Endpoint Mixed Norm Estimate of Calder\'on-Zygmund Operators}
\author{Zehan Hu}
\maketitle

\begin{abstract}

It is known that that the endpoint mixed norm estimate $|| \, ||Tf(x,y)||_{L_{x}^{p}}||_{L_{y}^{\infty}} \lesssim  || \, ||f(x,y)||_{L_{x}^{p}}||_{L_{y}^{\infty}}$ in general does not hold for Calder\'on-Zygmund operator $T$. In this article, we show that when $p=2$, even if we make the right hand side of the above estimate larger by replacing it with $ || \, ||e^{x^2+y^2}f(x,y)||_{L_{y}^{\infty}}||_{L_{x}^{\infty}} $, the estimate does not hold for the double Riesz transform given by the kernel $K(x,y)=\frac{xy}{2\pi(x^2+y^2)^{2}}$. As a consequence we will show that the mixed norm estimate $|| \, ||Tf(x,y)||_{L_x^{p}}||_{L_y^{\infty}} \lesssim|| \, ||f(x,y)||_{L_y^{\infty}}||_{L_x^{p}}$ does not hold for double Riesz transform and $p \geq 2$.

\end{abstract}

\section*{Introduction}

\newtheorem*{def1.1}{Definition 1.1}
\begin{def1.1}
$\quad$  A kernel $K: \mathbb{R}^n-\{0\} \to \mathbb{C}$ is called a Calder\'on-Zygmund kernel if there exists some constant $B$ such that
  \begin{enumerate}
      \item $|K(x)|<B|x|^{-n}$
      \item $\int_{|x|>2|y|} |K(x)-K(x-y)| dx \leq B$ for all $|y|>0$
      \item $\int_{r<|x|<s} K(x) dx =0$ for $0<r<s<\infty$.
  \end{enumerate}
\end{def1.1}

\newtheorem*{def1.2}{Definition 1.2}
\begin{def1.2}  
For a Calder\'on-Zygmund kernel $K$ and $f\in \mathcal{S}(\mathbb{R}^n)$, let 
$$Tf(x)=\lim_{\epsilon \to 0}\int_{|x-y|>\epsilon} K(x-y)f(y) dy .$$
$T$ is called the Calder\'on-Zygmund operator with kernel $K$.
\end{def1.2}

The followings are two examples of Calder\'on-Zygmund operators. 

\newtheorem*{def1.3}{Example 1.3}
\begin{def1.3}
The Riesz transforms $R_j$'s are the Calder\'on-Zygmund operators given by kernels $K_j(x)$ $=\frac{x_j}{|x|^{n+1}} $
\end{def1.3}

\newtheorem*{def1.4}{Example 1.4}
\begin{def1.4}  
  The double Riesz transforms $R_{ij}$'s are the Calder\'on-Zygmund operators given by kernels \[
    K_{ij}(x)= 
    \begin{cases}
    C_n\frac{x_ix_j}{|x|^{n+2}}   & \text{if } i \not= j \\
    C_n\frac{x_i^2-n^{-1}|x|^2}{|x|^{n+2}}       &\text{if } i=j
    \end{cases}
  \]where $C_n$'s are some dimensional constants.
\end{def1.4}
	By computation, for Schwartz function $u$, 
	\begin{align}
	\frac{\partial^2 u}{\partial x_i \partial x_j} (x)=R_{ij}(\Delta u)(x)+\frac{1}{n} \delta_{ij} \Delta u(x) .
	\end{align}
    
So, in particular, the $L^p$ boundedness of $R_{ij}$ is equivalent to the following $L^p$ estimate: $$ ||\frac{\partial^2 u}{\partial x_i \partial x_j} (x)||_{L^p} \leq C||\Delta u(x) ||_{L^p}.$$
 
A Calder\'on-Zygmund operator $T$ can be extended to an bounded operator on $L^p(\mathbb{R}^n)$:

\newtheorem*{thm}{Theorem 1.5 (Calder\'on-Zygmund)}
  \begin{thm}
    For $T$ a Calder\'on-Zygmund operator and $f\in \mathcal{S}(\mathbb{R}^n)$,
    $$ ||Tf(x)||_{L^p} \leq C_p || f(x) ||_{L^p} $$ for $1<p<\infty$.
    Hence, $T$ can be extended to a bounded operator on $L^p(\mathbb{R}^n)$ for $1<p<\infty$. This fails when $p=1$ or $\infty$
  \end{thm}

It is also known that the following mixed norm estimate holds for Calder\'on-Zygmund operators given by Calder\'on-Zygmund kernels with some additional regularity assumptions. [1, p.448, Theorem 1].

	\newtheorem*{thm6}{Theorem 1.6 (Mixed norm estimates)}
  \begin{thm6}
    For $T$, a Calder\'on-Zygmund operator given by kernel $K$ defined on $\mathbb{R}^2-\{0\}$ satisfying additional regularity conditions
	\[ 
		|\partial_2 K(x,y)| \leq A (x^2+y^2)^{-\frac{3}{2}} \quad
		|\partial_1 \partial_2 K(x,y)| \leq A (x^2+y^2)^{-2} 
	\]
	we have 
    $$|| \, ||Tf(x,y)||_{L_{x}^{p}}||_{L_{y}^{q}} \leq C || \, ||f(x,y)||_{L_{x}^{p}}||_{L_{y}^{q}}$$ for $1<p,q<\infty$
  \end{thm6} 
  \newtheorem*{rmk4}{Remark}
  \begin{rmk4}
    The above theorem no longer holds when $q=1$ or $\infty$. For example, consider the the double Riesz transform given by the kernel $K(x,y)=\frac{xy}{2\pi(x^2+y^2)^{2}}$ and a sequence of compactly supported functions with supports shrinking to a point. Indeed, it can be easily verified that the kernel corresponding to $R_{12}$ satisfies the additional regularity conditions in Theorem 1.6.
  \end{rmk4}

  \newtheorem*{Question}{Question}
  \begin{Question}
    We are curious about what will happen if we make the right hand side of the above mixed norm estimate larger by changing the order of the two norms.
	That is, for $T$, a Calder\'on-Zygmund operator given by a kernel $K$ defined on $\mathbb{R}^2-\{0\}$, can we always find a constant $C_p$ such that we have $|| \, ||Tf(x,y)||_{L_x^{p}}||_{L_y^{\infty}} \leq C_p || \, ||f(x,y)||_{L_y^{\infty}}||_{L_x^{p}}$?
	
  \end{Question}

	Note that the sequence in Remark following Theorem 1.6 does not contradict the following new estimate $|| \, ||Tf(x,y)||_{L_x^{p}}||_{L_y^{\infty}} \lesssim || \, ||f(x,y)||_{L_y^{\infty}}||_{L_x^{p}}$. 

In the next section we will see that we cannot have such estimate for the double Riesz transform $R_{12}$ given by the kernel $K(x,y)=\frac{xy}{2\pi(x^2+y^2)^{2}}$ and $p = 2$. In fact, we will show something stronger: the estimate  $|| \, ||R_{12}f(x,y)||_{L_x^{2}}||_{L_y^{\infty}} \lesssim || \, ||e^{x^2+y^2}f(x,y)||_{L_y^{\infty}}||_{L_x^{\infty}}$ does not hold. Then we will show by interpolation that the mixed norm estimate $|| \, ||Tf(x,y)||_{L_x^{p}}||_{L_y^{\infty}} \lesssim|| \, ||f(x,y)||_{L_y^{\infty}}||_{L_x^{p}}$ does not hold for this double Riesz transform and $p \geq 2$.

\section*{Main Results}

We first give a counterexample to the estimate $|| \, ||Tf(x,y)||_{L_x^{p}}||_{L_y^{\infty}} \lesssim || \, ||e^{x^2+y^2}f(x,y)||_{L_y^{\infty}}||_{L_x^{\infty}}$ for $T=R_{12}$, a double Riesz transform, and $p=2$. Throughout the passage, we use the following convention of Fourier transform and inverse Fourier transform:

\begin{align*}
\mathcal{F}(f)(\xi)=\left( \frac{1}{\sqrt{2\pi}} \right)^n \int e^{-ix \cdot \xi} f(x) dx \\
\mathcal{F}^{-1}(f)(\xi)=\left( \frac{1}{\sqrt{2\pi}} \right)^n \int e^{ix \cdot \xi} f(x) dx 
\end{align*}
for function $f$ on $\mathbb{R}^n$.

\newtheorem*{theorem2}{Proposition 2.1 (A Counterexample)}
\begin{theorem2} 
Let $R_{12}$ be the double Riesz transform given by the kernel $K(x,y)=\frac{xy}{2\pi(x^2+y^2)^{2}}$. Then there exists a sequence of Schwartz functions $g_n$ such that the sequence 
$\{ || \, ||R_{12}g_n(x,y)||_{L_{x}^{2}}||_{L_{y}^{\infty}} \}$ goes to infinity, while the the sequence $\{ || \, ||e^{x^2+y^2}g_n(x,y)||_{L_{y}^{\infty}}||_{L_{x}^{\infty}} \}$ stays bounded.

\end{theorem2}

\begin{proof}
We first construct a sequence of Schwartz functions $f_j$'s as follows. Define $\chi_1^{(j)}$ to be a smooth bump function supported on $[2^{-j}, 2^{-j+1}]$ such that $\chi_1^{(j)}$ equals to $1$ on $[2^{-j}+2^{-j-10}, 2^{-j+1}-2^{-j-10}]$. Let $\chi_2(x)=e^{-x^2}$. Define $\chi$ to be a smooth bump function supported on $[0,A]$ for $3 \leq A \leq 100$ such that $\chi$ equals $1$ on $[\frac{1}{4}, A-\frac{1}{4}]$. Define $\chi_3^{(j)}=\chi(x-2^j)+\chi(-x+2^j)$. Now we set $f_j(x,y)=\chi_1^{(j)}(y)\chi_2(x)\mathcal{F}^{-1}(\chi_3^{(j)})(x)$. Now, for $n>100$, define $g_n=\sum_{j=100}^n f_j$.

Next, we show that $\{ || \, ||e^{x^2+y^2}g_n(x,y)||_{L_{y}^{\infty}}||_{L_{x}^{\infty}} \}$ and $\{ || \, ||g_n(x,y)||_{L_{y}^{\infty}}||_{L_{x}^{2}} \}$ are bounded by some fixed constant.

\begin{align*}
|\mathcal{F}^{-1}(\chi_3^{(j)})(x)| & \leq  |\mathcal{F}^{-1}(\chi(\xi-2^j))(x)|+ |\mathcal{F}^{-1}(\chi(-\xi+2^j))(x)| \\
& = |\frac{1}{2\pi}\int e^{ i x \cdot \xi} \chi(\xi-2^j) d\xi|+|\frac{1}{2\pi}\int e^{ i x \cdot \xi} \chi(-\xi+2^j) d\xi| \\
&\leq \sup|\mathcal{F}^{-1}(\chi)| + \sup|\mathcal{F}(\chi)| .
\end{align*}

Let $D = \sup|\mathcal{F}^{-1}(\chi)| + \sup|\mathcal{F}(\chi)|$. Thus,
$$|f_j(x,y)| \leq |e^{-x^2} \mathcal{F}^{-1}(\chi_3^{(j)})(x)| \leq De^{-x^2}.$$

Since $f_j$'s are supported on $[2^{-j}, 2^{-j+1}]$ in $y$ and $|e^{y^2}\chi_1^{(j)}(y)| \leq e$ for $j \geq 100$, we have 

$$ || e^{x^2} g_n(x,y) ||_{L_{y}^\infty} \leq \sup_{100\leq j \leq n} || e^{x^2} f_j(x,y) ||_{L_{y}^\infty} \leq \sup_{100\leq j \leq n} | \mathcal{F}^{-1}(\chi_3^{(j)})(x)| \leq D.$$

$$ || e^{x^2+y^2} g_n(x,y) ||_{L_{y}^\infty} \leq \sup_{100\leq j \leq n} || e^{x^2+y^2} f_j(x,y) ||_{L_{y}^\infty} \leq \sup_{100\leq j \leq n} | e \mathcal{F}^{-1}(\chi_3^{(j)})(x)| \leq De.$$

Therefore, $||\,|| e^{x^2+y^2} g_n(x,y) ||_{L_{y}^\infty}||_{L_{x}^\infty}$ is bounded.

In particular, 
$$ ||\,|| g_n(x,y) ||_{L_{y}^\infty} ||_{L_{x}^2} \leq || e^{-x^2} ||_{L_{x}^2} \, ||\, || e^{x^2} g_n(x,y) ||_{L_{y}^\infty}||_{L_{x}^\infty} \\ \leq D|| e^{-x^2} ||_{L_{x}^2}.$$

Therefore,  $\{ || \, ||g_n(x,y)||_{L_{y}^{\infty}}||_{L_{x}^{2}} \}$ is bounded.

In order to show that $\{ || \, ||R_{12}g_n(x,y)||_{L_{x}^{2}}||_{L_{y}^{\infty}} \}$ goes to infinity, it suffices to show that $\{ ||R_{12}g_n(x,0)||_{L_{x}^{2}} \}$ goes to infinity. 
\begin{align*}
|| R_{12}g_n(x,0) ||_{L_{x}^2}^2 &= \int \left| \sum_{i=100}^n R_{12}f_i(x,0) \right|^2 dx=\int \left| \sum_{i=100}^n \mathcal{F}_1(R_{12}f_i)(\xi_1,0) \right|^2 d\xi_1 \\
& \geq \sum_{j=100}^n \int_{2^j+1}^{2^j+A-1}  \left| \sum_{i=100}^n \mathcal{F}_1(R_{12}f_i)(\xi_1,0) \right|^2 d\xi_1.
\end{align*}

We begin with computing $|\mathcal{F}_1(R_{12}f)(\xi_1,y)|$ as follows, where $\mathcal{F}_{i}$ is the fourier transform with respect to the $i^{th}$ coordinate. 
$$|\mathcal{F}_1(R_{12}f)(\xi_1,0)| =  \left| \mathcal{F}_2^{-1} \left( \ \frac{\xi_1\xi_2}{|\xi|^2} \mathcal{F}f  \right) (\xi_1,0) \right| = C \left| \int \mathcal{F}_1(f)(\xi_1,z)K(\xi_1, -z) dz \right|, $$
where $C>0$ is some constant, and
$$ K(\xi_1, z)= \frac{1}{\sqrt{2\pi}}\int e^{i z \xi_2} \frac{\xi_1\xi_2}{|\xi|^2} d\xi_2= \frac{1}{\sqrt{2\pi}} \int e^{i \xi_1 z \eta} \frac{\eta}{1+\eta^2}\xi_1 d\eta = \xi_1  \mathcal{F}_{\eta}^{-1}(\frac{\eta}{1+\eta^2}) (\xi_1z) . $$ 

Thus, 
\begin{align*}
|\mathcal{F}_1(R_{12}f_j)(\xi_1,0)| &= C \left| \int \mathcal{F}_1(f_j)(\xi_1,z) \xi_1  \mathcal{F}_{\eta}^{-1}(\frac{\eta}{1+\eta^2}) (-\xi_1z) dz \right| \\
&= C\sqrt{\frac{\pi}{2}} \left| \int \mathcal{F}_1(f_j)(\xi_1,z) \xi_1 e^{-|\xi_1||z|}  dz \right| \\
&= C\sqrt{\frac{\pi}{2}} \left| \left(\mathcal{F}(\chi_2) \ast \chi_3^{(j)}  \right)(\xi_1) \right| \, |\xi_1| \, \left| \int \chi_1^{(j)}(z) e^{-|\xi_1||z|}  dz  \right|.
\end{align*}

Note that by our construction of $\chi_1^{(j)}$, 
$$ \int_{2^{-j}+2^{-j-10}}^{2^{-j+1}-2^{-j-10}} e^{-|\xi_1|z} dz \leq \left| \int \chi_1^{(j)}(z) e^{-|\xi_1||z|}  dz  \right| \leq \int_{2^{-j}}^{2^{-j+1}} e^{-|\xi_1|z}  dz.$$

Now we estimate $\int_{2^j+1}^{2^j+A-1}  \left| \sum_{i=100}^n \mathcal{F}_1(R_{12}f_i)(\xi_1,0) \right|^2 d\xi_1$. To this end, we estimate $\left|\mathcal{F}_1(R_{12}f_j)(\xi_1,0) \right|$ from below for $\xi_1 \in [2^j+1, 2^j+A-1]$. 

\begin{align*}
|\mathcal{F}_1(R_{12}f_j)(\xi_1,0)| &\geq C\sqrt{\frac{\pi}{2}} \left| \left(\mathcal{F}(\chi_2) \ast \chi_3^{(j)}  \right)(\xi_1)\right| \, \left|  e^{-|\xi_1|(2^{-j}+2^{-j-10})} -e^{-|\xi_1|(2^{-j+1}-2^{-j-10})} \right| \\
&\geq C\sqrt{\frac{\pi}{2}} (e^{-2-2^{-9}}-e^{-4+2^{-9}})\left| \left( \mathcal{F}(\chi_2) \ast \chi_3^{(j)}  \right)(\xi_1) \right| \\
&\geq C\frac{\sqrt{\pi}}{2} (e^{-2-2^{-9}}-e^{-4+2^{-9}}) \int_{-\frac{1}{4}}^{\frac{1}{4}} e^{-\frac{\xi^2}{4}}\chi_3^{(j)}(\xi_1-\xi) d\xi \\
&\geq C\frac{\sqrt{\pi}}{2} (e^{-2-2^{-9}}-e^{-4+2^{-9}}) \int_{-\frac{1}{4}}^{\frac{1}{4}} e^{-\frac{\xi^2}{4}} d\xi,
\end{align*} 
where the second inequality is due to the fact that $e^{-|x|(2^{-j}+2^{-j-10})} -e^{-|x|(2^{-j+1}-2^{-j-10})}$ is decreasing on $ [2^j+1, \infty)$, and the last inequality is due to our construction of $\chi_3^{(j)}$.

Set $E=\frac{\sqrt{\pi}}{2} (e^{-2-2^{-9}}-e^{-4+2^{-9}}) \int_{-\frac{1}{4}}^{\frac{1}{4}} e^{-\frac{\xi^2}{4}} d\xi>0$. 

Next, we estimate $\left|\mathcal{F}_1(R_{12}f_i)(\xi_1,0) \right|$ from above for $\xi_1 \in [2^j+1, 2^j+A-1]$ and $i<j$.
\begin{align*}
\left|\mathcal{F}_1(R_{12}f_i)(\xi_1,0) \right| & \leq C \sqrt{\frac{\pi}{2}} \left| \left(\mathcal{F}(\chi_2) \ast \chi_3^{(i)}  \right)(\xi_1) \left( e^{-|\xi|2^{-i}} -e^{-|\xi|2^{-i+1}} \right) \right| \\
& \leq \frac{C}{4}\sqrt{\frac{\pi}{2}} \left| \left(\mathcal{F}(\chi_2) \ast \chi_3^{(i)}  \right)(\xi_1) \right| \\
& \leq \frac{C}{4}\frac{\sqrt{\pi}}{2} \int_{2^i}^{2^i+A} e^{-\frac{(\xi_1-\xi)^2}{4}} d\xi + \frac{C}{4}\frac{\sqrt{\pi}}{2} \int_{-2^i-A}^{-2^i} e^{-\frac{(\xi_1-\xi)^2}{4}} d\xi \\
& \leq C \frac{\sqrt{\pi}}{2} \int_{2^i}^{2^i+A} (\xi_1-\xi)^{-2} d\xi + C\frac{\sqrt{\pi}}{2} \int_{-2^i-A}^{-2^i} (\xi_1-\xi)^{-2} d\xi \\
& = C\frac{\sqrt{\pi}}{2} \left( \frac{1}{\xi_1-2^i-A} - \frac{1}{\xi_1-2^i} + \frac{1}{\xi_1+2^i}-\frac{1}{\xi_1+2^i+A} \right) \\
& \leq C\sqrt{\pi} \frac{1}{\xi_1-2^i-A} \leq 4C\sqrt{\pi}2^{-j}.
\end{align*}

Next, we estimate $\left|\mathcal{F}_1(R_{12}f_i)(\xi_1,0) \right|$ from above for $\xi_1 \in [2^j+1, 2^j+A-1]$ and $i>j$. By the same calculation, 
\begin{align*}
\left|\mathcal{F}_1(R_{12}f_i)(\xi_1,0) \right| &\leq C\frac{\sqrt{\pi}}{2} \left( \frac{1}{\xi_1-2^i-A} - \frac{1}{\xi_1-2^i} + \frac{1}{\xi_1+2^i}-\frac{1}{\xi_1+2^i+A} \right) \\
& \leq  C\sqrt{\pi} \frac{1}{2^i-\xi_1} \leq 4C\sqrt{\pi}2^{-i}.
\end{align*}

Now we estimate $ \left| \sum_{i=100}^n \mathcal{F}_1(R_{12}f_i)(\xi_1,0) \right|$ from below for $\xi_1 \in [2^j+1, 2^j+A-1]$. 
\begin{align*}
\left| \sum_{i=100}^n \mathcal{F}_1(R_{12}f_i)(\xi_1,0) \right| &\geq \left| \mathcal{F}_1(R_{12}f_j)(\xi_1,0) \right| - \sum_{100\leq i<j} \left|  \mathcal{F}_1(R_{12}f_i)(\xi_1,0) \right|-\sum_{j< i\leq N} \left|  \mathcal{F}_1(R_{12}f_i)(\xi_1,0) \right| \\
& \geq CE - (j-100)4C\sqrt{\pi}2^{-j} - 4C\sqrt{\pi}2^{-j} \\
& \geq CE - 4C\sqrt{\pi}j2^{-j} \geq C(E - 4\sqrt{\pi}(100) 2^{-100}) >0.
\end{align*}

Hence, 
\begin{align*}
||R_{12}g_n(x,0)||_{L_{x}^{2}}^2 &\geq \sum_{j=100}^n \int_{2^j+1}^{2^j+A-1}  \left| \sum_{i=100}^n \mathcal{F}_1(R_{12}f_i)(\xi_1,0) \right|^2 d\xi_1 \\
& \geq \sum_{j=100}^n \int_{2^j+1}^{2^j+A-1} (CE - 4C\sqrt{\pi}j2^{-j})^2 d\xi_1 \\
& \geq \sum_{j=100}^n (A-2) (CE - 4C\sqrt{\pi}(100)2^{-100})^2 \\
& \geq (n-100)(A-2) (CE - 4C\sqrt{\pi}(100)2^{-100})^2 .
\end{align*}

Therefore, $\{ ||R_{12}g_n(x,0)||_{L_{x}^{2}} \}$ goes to infinity and so does $\{ || \, ||R_{12}g_n(x,y)||_{L_{x}^{2}}||_{L_{y}^{\infty}} \}$.

\end{proof}

Therefore, we do not have the estimate $|| \, ||R_{12}f(x,y)||_{L_x^{2}}||_{L_y^{\infty}} \lesssim || \, ||e^{x^2+y^2}f(x,y)||_{L_y^{\infty}}||_{L_x^{\infty}}$ or the estimate $|| \, ||R_{12}f(x,y)||_{L_x^{2}}||_{L_y^{\infty}} \lesssim || \, ||f(x,y)||_{L_y^{\infty}}||_{L_x^{2}}$ for the double Riesz transform $R_{12}$ defined above. 

Next, we show that the estimate $|| \, ||R_{12}f(x,y)||_{L_x^{p}}||_{L_y^{\infty}} \lesssim || \, ||f(x,y)||_{L_y^{\infty}}||_{L_x^{p}}$ does not hold for $p>2$. To this end, we prove the following interpolation theorem and a mixed weak $L^1$ estimate. 

\newtheorem*{theorem*}{Proposition 2.2 (An Interpolation theorem)}
\begin{theorem*}

Let T be a sublinear operator from $L^{p_0}(\mathbb{R},L^{\infty}(\mathbb{R})) + L^{p_1}(\mathbb{R},L^{\infty}(\mathbb{R}))$ to the space of measurable functions. Assume that for $p_0<p_1<\infty$ we have 
$$|| \, ||Tf(x,y)||_{L_x^{p_0,\infty}}||_{L_y^{\infty}} \leq A_0 || \, ||f(x,y)||_{L_y^{\infty}}||_{L_x^{p_0}}$$
$$|| \, ||Tf(x,y)||_{L_x^{p_1,\infty}}||_{L_y^{\infty}} \leq A_1 || \, ||f(x,y)||_{L_y^{\infty}}||_{L_x^{p_1}}.$$
Then for $p_0<p<p_1$, we have 
$$|| \, ||Tf(x,y)||_{L_x^{p}}||_{L_y^{\infty}} \leq C(p_0,p_1,p,A_0,A_1) || \, ||f(x,y)||_{L_y^{\infty}}||_{L_x^{p}}$$
where $C(p_0,p_1,p,A_0,A_1)$ is a constant depending on $p_0,p_1,p,A_0,A_1$.
\end{theorem*}

\begin{proof}
$\quad$ Pick any $p_0<p<p_1$. For any $f \in L^p(\mathbb{R},L^\infty(\mathbb{R}))$, we have $||\,||f(x,y)||_{L_y^{\infty}}||_{L_x^{p}} <\infty$. Now we decompose $f$ into a sume of two functions defined as follows:
\[
    f_0^\alpha(x,y)= 
\begin{cases}
    f(x,y)& \text{for } |f(x,y)| > \alpha \\
    0              &\text{for } |f(x,y)| \leq \alpha
\end{cases}
\]
\[
    f_1^\alpha(x,y)= 
\begin{cases}
    f(x,y)& \text{for } |f(x,y)| \leq \alpha \\
    0              &\text{for } |f(x,y)| > \alpha.
\end{cases}
\]
By our definition, $f=f_0^\alpha+f_1^\alpha$.

Next, we note that $f_0^\alpha(\cdot, y) \in L^{p_0}(\mathbb{R})$ and $f_1^\alpha(\cdot, y) \in L^{p_1}(\mathbb{R})$ for $a.e. \, y$. Indeed, for $a.e.\, y$, we have 
\begin{align*}
||f_0^\alpha(\cdot, y)||_{L_x^{p_0}}^{p_0} &= \int_{\{x:|f(x,y)|>\alpha\}} |f(x,y)|^{p_0} dx = \int_{\{x:|f(x,y)|>\alpha\}} |f(x,y)|^{p} |f(x,y)|^{p_0-p} dx \\
&\leq (\alpha)^{p_0-p} \int_{\{x:|f(x,y)|>\alpha\}} ||f(x,\cdot)||_{L_y^{\infty}}^{p} dx \leq (\alpha)^{p_0-p}||\, ||f(x,y)||_{L_y^{\infty}}||_{L_x^{p}}
\end{align*}
and
\begin{align*}
||f_1^\alpha(\cdot, y)||_{L_x^{p_1}}^{p_1}&= \int_{\{x:|f(x,y)| \leq \alpha\}} |f(x,y)|^{p_1} dx = \int_{\{x:|f(x,y)| \leq \alpha\}} |f(x,y)|^{p} |f(x,y)|^{p_1-p} dx \\
&\leq (\alpha)^{p_1-p} \int_{\{x:|f(x,y)| \leq \alpha\}} ||f(x,\cdot)||_{L_y^{\infty}}^{p} dx \leq (\alpha)^{p_1-p}||\, ||f(x,y)||_{L_y^{\infty}}||_{L_x^{p}}.
\end{align*}

Since $T$ is sublinear, $|T(f)|\leq|T(f_0^\alpha)|+|T(f_1^\alpha)|$. Hence, $d_{Tf(\cdot,y)}(\alpha) \leq d_{Tf_0^\alpha(\cdot,y)}(\frac{\alpha}{2}) + d_{Tf_1^\alpha(\cdot,y)}(\frac{\alpha}{2})$ for $a.e. \, y$. Therefore, for $a.e. \, y$, we have 
\begin{align*}
d_{Tf(\cdot,y)}(\alpha) &\leq d_{Tf_0^\alpha(\cdot,y)}(\frac{\alpha}{2}) + d_{Tf_1^\alpha(\cdot,y)}(\frac{\alpha}{2}) \\
& =(\frac{\alpha}{2})^{-p_0} \left( (\frac{\alpha}{2}) \left( d_{Tf_0^\alpha(\cdot,y)}(\frac{\alpha}{2}) \right)^{\frac{1}{p_0}}   \right)^{p_0}+(\frac{\alpha}{2})^{-p_1} \left( (\frac{\alpha}{2}) \left( d_{Tf_1^\alpha(\cdot,y)}(\frac{\alpha}{2}) \right)^{\frac{1}{p_1}}   \right)^{p_1} \\
& \leq (\frac{\alpha}{2})^{-p_0} ||Tf_0^\alpha (\cdot, y) ||_{L_x^{p_0,\infty}}^{p_0} + (\frac{\alpha}{2})^{-p_1} ||Tf_1^\alpha (\cdot, y) ||_{L_x^{p_1,\infty}}^{p_1} \\
& \leq (\frac{\alpha}{2})^{-p_0} ||\,||Tf_0^\alpha (x, y) ||_{L_x^{p_0,\infty}} ||_{L_y^\infty}^{p_0} + (\frac{\alpha}{2})^{-p_1} ||\,||Tf_1^\alpha (x, y) ||_{L_x^{p_1,\infty}}||_{L_y^\infty}^{p_1} \\
& \leq A_0^{p_0}(\frac{\alpha}{2})^{-p_0} ||\,||f_0^\alpha (x, y) ||_{L_y^\infty} ||_{L_x^{p_0}}^{p_0} + A_1^{p_1}(\frac{\alpha}{2})^{-p_1} ||\,||f_1^\alpha (x, y) ||_{L_y^\infty}||_{L_x^{p_1}}^{p_1} .
\end{align*}
Now we will compute $||\,||f_0^\alpha (x, y) ||_{L_y^\infty} ||_{L_x^{p_0}}^{p_0}$ and $||\,||f_1^\alpha (x, y) ||_{L_y^\infty}||_{L_x^{p_1}}^{p_1}$ in terms of $f$.

First, we observe that 

\begin{align*}
||f(x,y)||_{L_y^{\infty}} \leq \alpha & \Rightarrow \text{ for } a.e.\, y\quad |f(x,y)| \leq \alpha \\
& \Rightarrow \text{ for } a.e. \,y \quad f_0^\alpha(x,y)=0 \\
& \Rightarrow ||f_0^\alpha(x,y)||_{L_y^\infty}=0
\end{align*}

and 

\begin{align*}
||f(x,y)||_{L_y^{\infty}} > \alpha & \Rightarrow \text{ for } a.e.\, y\quad |f(x,y)| > \alpha \\
& \Rightarrow \text{ for } a.e. \,y \quad f_0^\alpha(x,y)=f(x,y) \\
& \Rightarrow ||f_0^\alpha(x,y)||_{L_y^\infty}= ||f(x,y)||_{L_y^\infty}.
\end{align*}

Hence, 

\begin{align*}
||\,||f_0^\alpha (x, y) ||_{L_y^\infty} ||_{L_x^{p_0}}^{p_0} &= \int_{\mathbb{R}} || f_0^\alpha(x,y) ||_{L_y^\infty}^{p_0} dx =\int_{\{x: ||f(x,y)||_{L_y^\infty}>\alpha\}} ||f(x,y)||_{L_y^\infty}^{p_0} dx.
\end{align*}

Similarly, we get $$ ||\,||f_1^\alpha (x, y) ||_{L_y^\infty}||_{L_x^{p_1}}^{p_1} = \int_{\{ x: ||f(x,y)||_{L_y^\infty} \leq \alpha \}} ||f(x,y)||_{L_y^\infty}^{p_1} dx .$$

With $||\,||f_0^\alpha (x, y) ||_{L_y^\infty} ||_{L_x^{p_0}}^{p_0}$ and $||\,||f_1^\alpha (x, y) ||_{L_y^\infty}||_{L_x^{p_1}}^{p_1}$ computed, we can now estimate $d_{Tf(\cdot, y)}(\alpha)$ by $||f(x,y)||_{L_y^\infty}$. 

\begin{align*}
d_{Tf(\cdot, y)}(\alpha) &\leq (\frac{\alpha}{2})^{-p_0} {A_0}^{p_0} ||\,||Tf_0^\alpha (x, y) ||_{L_x^{p_0,\infty}} ||_{L_y^\infty}^{p_0} + (\frac{\alpha}{2})^{-p_1} {A_1}^{p_1}||\,||Tf_1^\alpha (x, y) ||_{L_x^{p_1,\infty}}||_{L_y^\infty}^{p_1} \\
& \leq (\frac{\alpha}{2})^{-p_0} {A_0}^{p_0} \int_{\{x: ||f(x,y)||_{L_y^\infty}>\alpha\}} ||f(x,y)||_{L_y^\infty}^{p_0} dx \\
& + (\frac{\alpha}{2})^{-p_1} {A_1}^{p_1} \int_{\{ x: ||f(x,y)||_{L_y^\infty} \leq \alpha \}} ||f(x,y)||_{L_y^\infty}^{p_1} dx.
\end{align*}

Finally, we estimate $|| Tf(\cdot, y) ||_{L_x^p}^p$ as follows:

\begin{align*}
||Tf(\cdot, y) ||_{L_x^p}^p & = p\int_0^\infty \alpha^{p-1} d_{Tf(\cdot, y)}(\alpha) d\alpha \\
& \leq p(2A_0)^{p_0}\int_0^\infty \alpha^{p-p_0-1} \int_{\{x: ||f(x,y)||_{L_y^\infty}>\alpha\}} ||f(x,y)||_{L_y^\infty}^{p_0} dx \, d\alpha \\
& + p(2A_1)^{p_1} \int_0^\infty \alpha^{p-p_1-1} \int_{\{ x: ||f(x,y)||_{L_y^\infty} \leq \alpha \}} ||f(x,y)||_{L_y^\infty}^{p_1} dx\, d\alpha \\
& = p(2A_0)^{p_0} \int_{\mathbb{R}} ||f(x,y)||_{L_y^\infty}^{p_0} \int_0^{||f(x,y)||_{L_y^\infty}} \alpha^{p-p_0-1} d\alpha \, dx \\
& + p(2A_1)^{p_1} \int_{\mathbb{R}} ||f(x,y)||_{L_y^\infty}^{p_1} \int_{||f(x,y)||_{L_y^\infty}}^\infty \alpha^{p-p_1-1} d\alpha \, dx \\
& = p\left( \frac{(2A_0)^{p_0}}{p-p_0} + \frac{(2A_1)^{p_1}}{p_1-p} \right) ||\,||f(x,y)||_{L_y^{\infty}}||_{L_x^p}^p
\end{align*}
for $a.e. \, y$.

Therefore, we conclude that $$ || \, ||Tf(x,y)||_{L_x^{p}}||_{L_y^{\infty}} \leq C(p_0,p_1,p,A_0,A_1) || \, ||f(x,y)||_{L_y^{\infty}}||_{L_x^{p}} . $$

\end{proof}

\newtheorem*{theorem}{Proposition 2.3}
\begin{theorem}
For $T$, a Calder\'on-Zygmund operator given by a kernel $K$ defined on $\mathbb{R}^2-\{0\}$ such that
$$|| \, ||Tf(x,y)||_{L_x^{p}}||_{L_y^{\infty}} \leq C || \, ||f(x,y)||_{L_y^{\infty}}||_{L_x^{p}}$$
we have 

$$|| \, ||Tf(x,y)||_{L_x^{1,\infty}}||_{L_y^{\infty}} \leq D || \, ||f(x,y)||_{L_y^{\infty}}||_{L_x^{1}}$$

\end{theorem}

\begin{proof}

It suffices to prove the inequality for $f \in \mathcal{S}(\mathbb{R}^2)$. Fix any $\alpha>0$. We apply Calder\'on-Zygmund decomposition to $||f(x,y)||_{L_y^\infty}$, i.e. we can find a collection of intervals $\mathcal{B}=\{ Q \}$ such that 

\begin{align*}
&|\bigcup_{\mathcal{B}} Q| \leq \frac{1}{\alpha} ||\,|| f(x,y)||_{L_y^\infty}||_{L_x^1} \quad \alpha < \frac{1}{|Q|} \int_Q ||f(x,y)||_{L_y^\infty} dx \leq 2\alpha \\
&||f(x,y)||_{L_y^\infty}=\tilde g(x)+\tilde b(x)
\end{align*}
where $\tilde b(x)=\sum_{\mathcal{B}} \chi_{Q}||f(x,y)||_{L_y^\infty} $ and $|\tilde g(x)|\leq \alpha$.

Now we define 
\[
b(x,y) \coloneqq \sum_{Q \in \mathcal{B}} \chi_{Q}(x)f(x,y)
\]
\[
g(x,y) \coloneqq f(x,y)-b(x,y)
\]
\[
f_1(\cdot,y)\coloneqq g(\cdot,y)+\sum_{Q\in \mathcal{B}} \chi_Q \left( \frac{1}{|Q|} \int_Q f(x,y) dx \right)
\]
\[
f_2(\cdot,y)\coloneqq b(\cdot,y)-\sum_{Q\in \mathcal{B}} \chi_Q \left( \frac{1}{|Q|} \int_Q f(x,y) dx \right).
\]

Hence we have $f=f_1+f_2$. Define $f_Q(x,y) = \chi_Q\left(f(x,y) -\frac{1}{|Q|}\int_Q f(x,y)dx \right)$. Thus, $f_2(x,y)=\sum_{\mathcal{B}}f_Q(x,y)$ and $\int f_Q(x,y) dx =0 $. Also, we have
\begin{align*}
||f_1(x,y)||_{L_y^\infty} &\leq ||g(x,y)||_{L_y^\infty}+\sum_{\mathcal{B}} \chi_Q \frac{1}{|Q|} \int_Q ||f(x,y)||_{L_y^\infty} dx \\
&\leq \tilde g(x)+2\alpha \sum_{\mathcal{B}}\chi_Q.
\end{align*}
Thus, $|| \,||f_1(x,y)||_{L_y^\infty}||_{L_x^\infty} \leq 2\alpha$. Moreover,
\begin{align*}
|| \,||f_1(x,y)||_{L_y^\infty}||_{L_x^1} &\leq \int|\tilde g(x)|dx + \sum_{\mathcal{B}} \int_Q ||f(x,y)||_{L_y^\infty}dx \\
&\leq || \,||f(x,y)||_{L_y^\infty}||_{L_x^1}.
\end{align*}

Now we estimate $d_{Tf(\cdot,y)}(\alpha)$ as follows:
\begin{align*}
d_{Tf(\cdot,y)}(\alpha) &= |\{ x \in \mathbb{R}: |Tf(x,y)|>\alpha \}| \\
&\leq |\{ x \in \mathbb{R}: |Tf_1(x,y)|>\frac{\alpha}{2} \}| +|\{ x \in \mathbb{R}: |Tf_2(x,y)|>\frac{\alpha}{2} \}|.
\end{align*}

We first estimate $|\{ x \in \mathbb{R}: |Tf_1(x,y)|>\frac{\alpha}{2} \}|$ using Chebyshev's Inequality, what we observed above, and our assumption on $T$:
\begin{align*}
|\{ x \in \mathbb{R}: |Tf_1(x,y)|>\frac{\alpha}{2} \}| &\leq \frac{2^p}{\alpha^p} || Tf_1(\cdot,y) ||_{L_x^p}^p \\
& \leq \frac{2^p}{\alpha^p}||\, || Tf_1(x,,y) ||_{L_x^p}||_{L_y^\infty} ^p \\
& \leq \frac{2^pC^p}{\alpha^p}||\, || f_1(x,y) ||_{L_y^\infty}||_{L_x^p} ^p\\
& =  \frac{2^pC^p}{\alpha^p} ||\, || f_1(x,y) ||_{L_y^\infty}||_{L_x^\infty}^{p-1} \: ||\, || f_1(x,y) ||_{L_y^\infty}||_{L_x^1} \\
& \leq \frac{2^{2p-1}C^p}{\alpha} ||\, || f(x,y) ||_{L_y^\infty}||_{L_x^1} .
\end{align*}

Next, we estimate $|\{ x \in \mathbb{R}: |Tf_2(x,y)|>\frac{\alpha}{2} \}|$. For each $Q$, let $Q^*$ be the interval such that $Q$ and $Q^*$ have the same center, and $|Q^*|=2|Q|$.

\begin{align*}
|\{ x \in \mathbb{R}: |Tf_2(x,y)|>\frac{\alpha}{2} \}| & \leq |\cup_{\mathcal{B}} Q^*|+ |\{ x \in \mathbb{R}-\cup_{\mathcal{B}} Q^* : |Tf_2(x,y)|>\frac{\alpha}{2} \} | \\
&\leq 2\sum_{\mathcal{B}} |Q| + \frac{2}{\alpha} \int_{ \mathbb{R}-\cup_{\mathcal{B}} Q^*} |Tf_2(x,y)|dx \\
&\leq \frac{2}{\alpha} ||\, || f(x,y) ||_{L_y^\infty}||_{L_x^1} + \frac{2}{\alpha} \sum_{\mathcal{B}} \int_{\mathbb{R}-Q^*} |Tf_Q(x,y)|dx .
\end{align*}

We now estimate $\int_{\mathbb{R}-Q^*} |Tf_Q(x,y)|dx$. First, from $\int f_Q(x,y) dx =0 $ we conclude that for $x\in \mathbb{R}-Q^*$,

\begin{align*}
Tf_Q(x,y) & = \int_{\mathbb{R}} \int_Q K(x-s,y-t)f_Q(s,t) ds \,dt \\
&= \int_{\mathbb{R}} \int_Q \left( K(x-s,y-t) -K(x-x_Q,y-t) \right) f_Q(s,t) ds \,dt
\end{align*}
where $x_Q$ is the center of interval $Q$.

Recall that the Caler\'on-Zygmund kernal $K$ has the following property: there is a constant $B>0$ such that 
\[
\int_{\{ |x|>2|y| \}} |K(x)-K(x-y)|dx \leq B \text{ for all } y \not=0.
\]

Thus, by our selection of $Q^*$,

\begin{align*}
&\int_{\mathbb{R}-Q^*} |Tf_Q(x,y)|dx \\
& \leq \int_{\mathbb{R}-Q^*} \int_{\mathbb{R}} \int_Q | K(x-s,y-t) -K(x-x_Q,y-t) |\, |f_Q(s,t)| ds \,dt\, dx \\
& \leq \int_Q ||f_Q(s,y)||_{L_y^\infty} \left( \int_{\mathbb{R}-Q^*} \int_{\mathbb{R}} | K(x-s,y-t) -K(x-x_Q,y-t) | dt \,dx \right) \, ds \\
& \leq B \int_{Q} ||f_Q(s,y)||_{L_y^\infty} ds \\
& \leq 2B \int_{Q} ||f(x,y)||_{L_y^\infty} dx .
\end{align*}

Hence, 
\begin{align*}
|\{ x \in \mathbb{R}: |Tf_2(x,y)|>\frac{\alpha}{2} \}| &\leq \frac{2}{\alpha} ||\, || f(x,y) ||_{L_y^\infty}||_{L_x^1} + \frac{2}{\alpha} \sum_{\mathcal{B}^y} \int_{\mathbb{R}-Q^*} |Tf_Q(x,y)|dx \\
& \leq \frac{2}{\alpha} ||\, || f(x,y) ||_{L_y^\infty}||_{L_x^1} + \frac{4B}{\alpha} \sum_{\mathcal{B}^y}\int_{Q} ||f(x,\cdot)||_{L_y^\infty} dx \\
& \leq \frac{2+4B}{\alpha}||\, || f(x,y) ||_{L_y^\infty}||_{L_x^1}.
\end{align*}

Therefore, combining what we have proved now, for some $D>0$ and $a.e. \, y$,

\begin{align*}
d_{Tf(\cdot,y)}(\alpha) &= |\{ x \in \mathbb{R}: |Tf(x,y)|>\alpha \}| \\
&\leq |\{ x \in \mathbb{R}: |Tf_1(x,y)|>\frac{\alpha}{2} \}| +|\{ x \in \mathbb{R}: |Tf_2(x,y)|>\frac{\alpha}{2} \}| \\
&\leq \frac{D}{\alpha}||\, || f(x,y) ||_{L_y^\infty}||_{L_x^1}.
\end{align*}

Thus, we conclude that 

\[
|| \, ||Tf(x,y)||_{L_x^{1,\infty}}||_{L_y^{\infty}} \leq D || \, ||f(x,y)||_{L_y^{\infty}}||_{L_x^{1}}.
\]

\end{proof}

\newtheorem*{corollary}{Corollary 2.4}
\begin{corollary}

For Calder\'on-Zygmund operator $T$, the estimate $|| \, ||Tf(x,y)||_{L_x^{q}}||_{L_y^{\infty}} \leq C_q || \, ||f(x,y)||_{L_y^{\infty}}||_{L_x^{q}}$ implies the estimate $|| \, ||Tf(x,y)||_{L_x^{p}}||_{L_y^{\infty}} \leq C_p || \, ||f(x,y)||_{L_y^{\infty}}||_{L_x^{p}}$ for all $1<p\leq q$. 

\end{corollary}
\begin{proof}

This follows from Proposition 2.2 and Proposition 2.3 above.

\end{proof}

\newtheorem*{cor}{Theorem 2.5}
\begin{cor}

For $p>2$, there exists no $C_p$ such that $|| \, ||R_{12}f(x,y)||_{L_x^{p}}||_{L_y^{\infty}} \leq C_p || \, ||f(x,y)||_{L_y^{\infty}}||_{L_x^{p}}$ for all Schwartz function $f$. 

\end{cor}

\begin{proof}

Assume, for the sake of contradiction, we have such estimate for $R_{12}$ and $p>2$, then by the previous corollary, we would have such estimate for $R_{12}$ and $p=2$, which contradicts Proposition 2.1 above. 

\end{proof}

\newtheorem*{cor1}{Corollary 2.6}
\begin{cor1}

For $p\geq 2$, there exists no $C_p$ such that $|| \, ||\partial_x\partial_yu(x,y)||_{L_x^{p}}||_{L_y^{\infty}} \leq  C_p || \, ||\Delta u(x,y)||_{L_y^{\infty}}||_{L_x^{p}}$ for all Schwartz function $u$.

\end{cor1}

\begin{proof}

This proposition follows directly from Proposition 2.5 and equation (1).

\end{proof}

\section*{Acknowledgement}

$\quad$This article is the written report of a project at Stanford Undergraduate Research Institute in Mathematics(SURIM) in 2021. I would like to thank my mentor, Professor Jonathan Luk, who proposed the topic of this article. I am grateful of his guidance and help throughout the summer. I would also like to thank SURIM program for funding and support for the research and Dr Pawel Grzegrzolka for organizing this program.

\end{document}